\RequirePackage{rotating}
\documentclass[12pt]{amsart}
\usepackage[graphicx]{realboxes}
\usepackage{float}
\restylefloat{table}
\usepackage{placeins}

\usepackage{amsmath}
\usepackage{amssymb}
\usepackage{epsfig}
\usepackage{wasysym}
\usepackage{graphicx}
\usepackage{tikz}
\usepackage{tikz-cd}
\usepackage{bm}
\usepackage{xcolor}
\usepackage{listings}
\numberwithin{equation}{section}
\usepackage{extpfeil}
\usepackage{array}
\usepackage{adjustbox}
\usepackage{longtable}
\usepackage{makecell}
\usepackage[all]{xy}
\usepackage{longtable}
\usepackage{rotating}
\input xy
\xyoption{all}
\usepackage{multirow}

\usetikzlibrary{positioning}
\usepackage[colorlinks=false,urlbordercolor=white]{hyperref}
  
\tikzset{sgplattice/.style={inner sep=1pt,norm/.style={red!50!blue},char/.style={blue!50!black},
  lin/.style={black!50}},cnj/.style={black!50,yshift=-2.5pt,left=-1pt of #1,scale=0.5,fill=white}}

\DeclareFontFamily{U}{mathb}{\hyphenchar\font45}
\DeclareFontShape{U}{mathb}{m}{n}{
      <5> <6> <7> <8> <9> <10> gen * mathb
      <10.95> mathb10 <12> <14.4> <17.28> <20.74> <24.88> mathb12
      }{}
\DeclareSymbolFont{mathb}{U}{mathb}{m}{n}
\DeclareMathSymbol{\righttoleftarrow}{3}{mathb}{"FD}

\calclayout
\allowdisplaybreaks[3]

\theoremstyle{plain}
\newtheorem{prop}{Proposition}[section]

\newtheorem{coro}[prop]{Corollary}

\newtheorem{lemm}[prop]{Lemma}

\theoremstyle{definition}

\newtheorem{rema}[prop]{Remark}

\newtheorem{exam}[prop]{Example}

\def\cO{{\mathcal O}}

\def\cX{{\mathcal X}}

\def\sA{{\mathsf A}}
\def\sD{{\mathsf D}}

\def\fS{{\mathfrak S}}

\def\fS{{\mathfrak S}}

\def\bA{{\mathbb A}}

\def\bP{{\mathbb P}}

\def\bZ{{\mathbb Z}}
\def\bR{{\mathbb R}}

\def\bC{{\mathbb C}}

\def\rH{{\mathrm H}}

\def\Br{\mathrm{Br}}

\def\Pic{\mathrm{Pic}}

\def\Gal{\mathrm{Gal}}
\def\Aut{\mathrm{Aut}}

\def\PGL{\mathsf{PGL}}

\def\lim{\mathrm{lim}}

\def\Cl{\mathrm{Cl}}

\makeatother
\makeatletter

\begin{document}
\title[Real cubic threefolds]{Rationality of singular cubic threefolds over $\bR$}

\author[I. Cheltsov]{Ivan Cheltsov}
\address{Department of Mathematics, University of Edinburgh, UK}
\email{I.Cheltsov@ed.ac.uk}

\author[Y. Tschinkel]{Yuri Tschinkel}
\address{
  Courant Institute,
  251 Mercer Street,
  New York, NY 10012, USA
}

\email{tschinkel@cims.nyu.edu}

\address{Simons Foundation\\
160 Fifth Avenue\\
New York, NY 10010\\
USA}

\author[Zh. Zhang]{Zhijia Zhang}

\address{
Courant Institute,
  251 Mercer Street,
  New York, NY 10012, USA
}

\email{zhijia.zhang@cims.nyu.edu}

\date{\today}

\begin{abstract}
We study rationality  properties of real singular cubic threefolds. 
\end{abstract}

\maketitle

\section{Introduction}
\label{sect:intro}

In this note we focus on rationality properties of singular cubic threefolds over the real numbers. 
We are guided by results in equivariant birational geometry in \cite{CTZ}, \cite{CMTZ}, which elucidated
birationality constructions as well as obstructions to birationality, in the equivariant context.

One of the main problems in birational geometry is the (stable) rationality problem; see, e.g., 
\cite{MT}, \cite{CT-bir}, \cite{Chel}, \cite{Pro}, \cite{AP}, \cite{V}, \cite{Schr} for reports on milestones in this area. Classically, one was interested in algebraically closed ground fields, of characteristic zero, 
mostly the complex numbers $\bC$. However, many (stable) rationality constructions rely crucially on results over nonclosed fields, e.g., function fields of smaller-dimensional varieties: in \cite{BCT} this  is used to produce nonrational but stably rational threefolds over $\bC$, via a degree 4 Del Pezzo fibration over $\bP^1$,   
and in \cite{AHTV} rational cubic fourfolds, via a degree 6 Del Pezzo fibration over $\bP^2$. The survey 
\cite{CT-corps} contains many results concerning rationality over arbitrary fields.

Of particular interest are varieties with simple presentations, such as hypersurfaces in (weighted) projective spaces, or linear sections of Grassmannians. 
Among the simplest such varieties are quadric and cubic hypersurfaces. Rationality of quadrics is completely settled: a necessary and sufficient condition is the existence of a rational point. 
For cubics, we do not currently have such criteria. Smooth complex cubic threefolds $X_{\bC}\subset \bP^4$ are irrational \cite{CG}, while the singular ones are rational unless they are cones over smooth cubic curves. 

In recent years, there have been several results concerning rationality of geometrically rational varieties over $\bR$. For example, it is now known that rationality of smooth intersections of two quadrics $X_{2,2}\subset \bP^n$, for $n=5$, is governed by the existence of rational points and lines \cite{HT-quad}. For $n=6$, a necessary and sufficient condition is that $X(\bR)$ is connected \cite{HKT-quad}. 
The papers \cite{BW}, \cite{Ji-coni}, \cite{CTA} explore rationality of conic and quadric bundles over $\bR$.  

Here, we focus on rationality of singular cubic threefolds
$$
X=X_{\bR}\subset \bP^4
$$
over $\bR$. Every such $X$ has real points.  
We recall standard rationality constructions, in presence of distinguished loci, over $\bR$:
\begin{itemize}
    \item 
    Projection from a singular point gives a birational map to $\bP^3$,
    \item  Projection from a line and a disjoint plane gives a birational map $X\dashrightarrow \bP^2\times\bP^1$. 
\end{itemize}
In particular, if the cubic is not a cone and there exists a real singular point, then $X$ is rational over $\bR$. 
Hence, throughout the paper, we assume that $X$ is not a cone and that 
$$
X^{\mathrm{sing}}(\bR)=\emptyset.
$$ 
This implies that either the singularities of $X_{\bC}$ are isolated and their number $s=s(X_{\bC})$ is even, i.e., $s=2,4,6,8,10$, or that 
the singular locus has positive dimension. 
By \cite{Avilov-forms}, a real form of a cubic with 10 isolated singularities is rational over $\bR$ -- geometrically, there is only one
such cubic, the {\em Segre} cubic.

Let $\Gamma=\Gal(\bC/\bR)\simeq \bZ/2$ be the group generated by complex conjugation. 
In addition to an action of $\Gamma$ on special loci of $X$, such as the singular locus $X^{\mathrm{sing}}$, we have an action on the geometric Picard group 
$\Pic(\tilde{X}_{\bC})$ of the minimal resolution of singularities $\tilde{X}$ of $X$. 
We investigate (stable) rationality over $\bR$, using the following tools: 
\begin{itemize}
 \item {\em classification}: we compute all possibilities for the action of $\Gamma$, configurations of singularities, and normal forms;  we freely use the terminology and techniques of \cite{CTZ} and \cite{CMTZ};
    \item {\em topology}: if $X(\bR)$ is disconnected then $X$ is not stably rational; 
    \item {\em cohomology}: 
    if
    \begin{equation} 
    \label{eqn:H1}
    \rH^1(\Gamma, \Pic(\tilde{X}_{\bC}))\neq 0, 
    \end{equation}
    then $X$ is not stably rational; we refer to this as the 
    {\bf(H1)}-obstruction.  
\end{itemize}

We summarize the results: 
\begin{itemize}
    \item $s=2$: 
    \begin{itemize}
    \item $2\sA_1$, $2\sA_2$: not rational over $\bR$.
        \item $2\sA_3$ with no plane: criteria for disconnectedness of $X(\bR)$, and examples with $X(\bR)$ connected or disconnected. 
       \item $2\sA_3$ with a plane: rational if and only if there is a real line disjoint from  the plane.
        \item $2\sA_4$: criteria for disconnectedness, and examples with $X(\bR)$ connected or disconnected.
        \item $2\sA_5$: all not stably rational.
        \item $2\sD_4$: criteria for disconnectedness, examples with rational $X$, and examples with $X(\bR)$  disconnected. 
    \end{itemize}

    \item $s=4$: 
    \begin{itemize}
        \item $4\sA_1$ with no plane, $4\sA_2$, $2\sA_2+2\sA_1$ are birational over $\bR$; criteria for disconnectedness, and examples of connected and disconnected $X(\bR)$.
        \item $4\sA_1$ with a plane: rational if and only if there is a real line disjoint from  the plane. 
        \item $2\sA_3+2\sA_1$ with three planes: criteria for disconnectedness, 
        examples with rational $X$, and with $X(\bR)$  disconnected.
         \item $2\sA_3+2\sA_1$ with one plane: always rational.
        \item $2\sD_4+2\sA_1$: rational if and only if it contains three planes.
    \end{itemize}
    \item $s=6$: \begin{itemize}
        \item $6\sA_1$ with no plane: rational if and only if it contains a real cubic scroll.
         \item $6\sA_1$ with one plane: always rational.
         \item $6\sA_1$ with three planes: criteria for disconnectedness, examples with rational $X$, and with $X(\bR)$ disconnected.  
         \item $2\sA_2+4\sA_1$, $2\sA_3+4\sA_1$: always rational. 
    \end{itemize}
    \item $s=8$: rational if and only if it contains three planes. 
\end{itemize}

Geometrically, cubic threefolds with nonisolated singularities are of four types \cite{yokocub}, \cite{Allcub}: with $X^{\rm{sing}}$ consisting of a plane, a line,  a conic, or a twisted quartic. We show that all real forms of such cubics are rational, unless  $X^{\rm{sing}}$ is a conic. In this last case, we have examples with connected or disconnected $X(\bR)$. The rationality of such cubics is open, see, e.g., \cite{CTA}.

\

\noindent
{\bf Acknowledgments:} 
The first author was partially supported by the Leverhulme Trust grant RPG-2021-229.
The second author was partially supported by NSF grant 2301983.

\section{Cohomology}
\label{sect:coho}

Here, we investigate the {\bf (H1)}-obstruction for cubics with isolated singularities. 

\begin{prop}\label{prop:h1sum}
Let $X$ be a real cubic threefold with isolated singularities over $\bC$. We have
    $$
    \rH^1(\Gamma,\Pic(\tilde X_\bC))\ne 0
    $$
    if and only if one of the following holds:
   \begin{enumerate}
       \item $X_\bC$ has $2\sA_5$-singularities,
       \item $X_\bC$ has $2\sD_4+2\sA_1$-singularities, and only one real plane,
       \item $X_\bC$ has $6\sA_1$-singularities in linearly general position, and contains no normal cubic scrolls over $\bR$,
       \item $X_\bC$ has $8\sA_1$-singularities, and only one real plane.
   \end{enumerate}
\end{prop}

\begin{proof}
When $X_{\bC}$ has the indicated configurations of singularities, the proof follows from Propositions~\ref{prop:2a5h1}, \ref{prop:2d4+2a1H1}, \ref{prop:6a1h1}, and \ref{prop:8-rat}. In all other cases, the {\bf (H1)}-obstruction is trivial, as shown in \cite{CTZ}, \cite{CMTZ}. 
\end{proof}

We now list representative examples with nontrivial cohomology for each of the cases in Proposition~\ref{prop:h1sum}.

   \begin{exam} 
   \label{exam:2A5-coho}
    $2\sA_5$: 
    Over $\bC$, every such cubic is given by 
    $$
x_1x_2x_3+x_1x_4^2+x_2x_5^2+x_3^3+bx_3x_4x_5=0,\quad b^2\ne -4.
    $$
Over $\bR$, every such cubic is given by 
\begin{equation} 
\label{eqn:2a5}
X=\{(x_1^2 + x_2^2)x_3 + x_3^3 - 2x_1(x_4^2-x_5^2) - 4x_2x_4x_5 + bx_3(x_4^2+x_5^2)=0\},
\end{equation}
for some $b\in\bR$, with singular points
$$
[1:i:0:0:0],\quad [1:-i:0:0:0].
$$
Consider the projection from the real line through the singularities
$$
\pi: X\dashrightarrow\bP^2,\quad (x_1,x_2,x_3,x_4,x_5)\mapsto(x_3, x_4, x_5).
$$
The image $\pi(X(\bR))$, given by 
$$
\left\{x_3^2+\frac{b-\sqrt{b^2+4}}{2}x_4^2+\frac{b-\sqrt{b^2+4}}{2}x_5^2\leq 0\right\}\subset\bP^2,
$$
is connected; it is easy to see that $X(\bR)$ is also connected.

The Galois action on $\Pic(\tilde X_\bC)$ is isomorphic to the $C_2$-action considered in \cite[Proposition 5.12]{CMTZ}, which implies that
\begin{equation} 
\label{eqn:C}
\rH^1(\Gamma,\Pic(\tilde X_{\bC}))=\bZ/2\bZ,
\end{equation}
hence $X$ is not stably rational over $\bR.$ 
See Proposition~\ref{prop:2a5h1} for details. 
\end{exam}

\begin{exam} 
\label{exam:6A1}
$6\sA_1$: In the equivariant context, assuming that the singularities are in 
linearly general positions, we obtain nontrivial cohomology from any $C_2\subseteq \Aut(X)$, which does not fix a node \cite[Proposition 7.5]{CTZ}. 
Over $\bR$, we realize this by
\begin{multline*} 
(x_1^2+x_2^2)(x_3+2x_4+39x_5)+2x_1(x_3^2+x_4^2+4x_5^2+x_3x_5)+\\+2x_2(x_3x_4-x_3^2+x_4^2+4x_5^2)+x_4^2x_3+4x_5^2x_3-2x_3^2x_4+39x_3^2x_5=0,
\end{multline*} 
with singular points
$$
[0 : 0 : 0 : -2i : 1],\quad [0 : 0 : 0 : 2i : 1],\quad [0 : -1 : -i : 
1 : 0],
$$
$$
[0 : -1 : i : 1 : 0],\quad [-i : 1 : 0 : 0 : 0],\quad [i : 1 : 0 : 
0 : 0].
$$
The equation 
$$
x_3^2 + x_4^2 + 20x_3x_5 + 4x_5^2=0
$$
cuts out two complex conjugate irreducible divisors on $X_{\bC}$, so that complex conjugation acts nontrivially on the class group $\Cl(X_{\bC})$. Computation in \cite[Proposition 7.5]{CTZ} implies \eqref{eqn:C}; it 
follows that $X$ is not stably rational over $\bR$. 

The image in $\bP^2_{x_3,x_4,x_5}$ of the projection from the line passing through the last two singular points is connected, and given by 
$$
\{  (x_3^2 + x_4^2 + 20x_3x_5 + 4x_5^2)(x_3^2 + x_4^2 - \frac{77}{2}x_3x_5 + 4x_5^2)\geq 0
\}\subset \bP^2_{x_3,x_4,x_5},
$$
which shows that $X(\bR)$ is also connected.
\end{exam}

\begin{exam}
\label{exam:8A1}
$8\sA_1$:
In the equivariant context, $X_{\bC}$ is given by 
    $$
(ax_1+x_2+bx_3)x_4x_5+x_4(x_3^2-x_1^2)+x_5(x_3^2-x_2^2)=0,\quad a,b\in \bC,
    $$
    and cohomology from $C_2$ generated by 
    $$
    \mathbf{(x)}\mapsto (ax_5-x_1,x_4-x_2,x_3,x_4,x_5).
    $$    
\end{exam}

Over $\bR$, we can realize the corresponding 
$\Gamma$-action on cohomology on $X$ given by
\begin{multline*}
  a_1(x_1^2 + x_3^2)x_4 + x_4x_5^2+a_2(
    x_2^2x_5 + x_3^2x_5 + x_4^2x_5)+a_3
    x_3x_4x_5=0,
\end{multline*} 
 with $a_1,a_2,a_3\in\bR$ and singularities at
  $$
 [0 : -i : 0 : 1
: 0], [0 : i : 0 : 1 : 0],
 [-i : 0 : 0 : 0 : 1], [i : 0 : 0 : 0 : 1],
$$
$$ [-i : -i : 1 : 0 : 0], [i : -i : 1 : 0 : 0],
 [-i : 
i : 1 : 0 : 0], [i : i : 1 : 
0 : 0].
$$
A direct computation of $\Cl(X_{\bC})$ as in Proposition~\ref{prop:8-rat} shows that the $\Gamma$-action gives nontrivial cohomology \eqref{eqn:C}.

\begin{exam} 
\label{exam:2D4+2A1}
$2\sD_4+2\sA_1$:
In the equivariant context, $X_{\bC}$ is given by 
    $$  x_1x_2x_3+x_1x_2x_4+x_5^2(a_1x_3+a_2x_4)+a_3x_3x_4x_5=0,\quad a_1,a_2,a_3\in \bC^\times,
    $$
  with nontrivial cohomology from the involution
    $$
    \mathbf{(x)}\mapsto (x_2,x_1,x_4+\frac{a_1-a_2}{a_3}x_5,x_3-\frac{a_1-a_2}{a_3}x_5,x_5).
    $$
    Over $\bR$, we can realize this by 
     $$
     (x_1^2 + x_2^2)x_3 + x_5(x_3^2 +x_4^2) + (b_1x_3+b_2x_4)x_5^2 + ax_5^3=0
     $$
     for $a,b_1,b_2\in \bR$ with $b_2^2<4a.$
   The $\sD_4$-singularities are  
     $$
     [1:i:0:0:0],\quad [1:-i:0:0:0],
     $$
     and the $\sA_1$-singularities are
     $$
     [0:0:1:i:0],\quad [0:0:1:-i:0].
     $$
    In these cases, the $\Gamma$-action gives nontrivial cohomology \eqref{eqn:H1}, see Proposition~\ref{prop:2d4+2a1H1} for more details. 
\end{exam}

The following diagram shows specialization patterns, obtained in investigations of equivariant birationalities in \cite{CTZ}, \cite{CMTZ}. The red-labeled cases (i.e., cubic threefolds with indicated singularity types) admit forms over $\bR$ with nontrivial cohomology 
$
\rH^1(\Gamma, \Pic(\tilde{X}_{\bC}))
$
from a specific action of $\Gamma$; and general cubic threefolds, indicated in black and specializing to these, 
have trivial cohomology. 



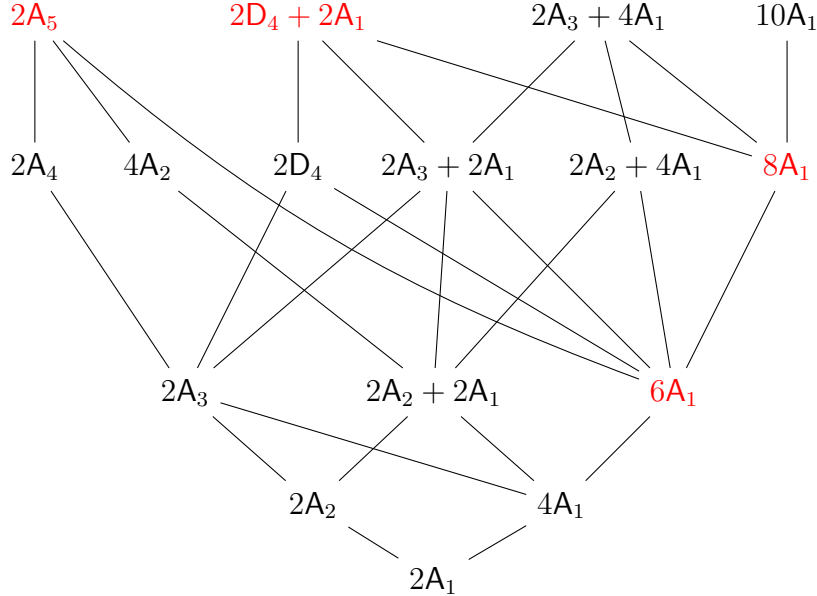
{
\[
\begin{tikzpicture}[commutative diagrams/every diagram]

\node (26) at (11,4)
	{$10\sA_1$};
\node (23) at (1,4)
	{\color{red}{$2\sA_5$}};

\node (22) at (4.5,4)
	{\color{red}$2\sD_4+2\sA_1$};
 \node (24) at (8.5,4)
	{$2\sA_3+4\sA_1$};

	\node (14) at (1,2)
	{$2\sA_4$};
	\node (18) at (2.5,2)
	{$4\sA_2$};
	
	\node (12) at (4.5,2)
	{$2\sD_4$};
	\node (15) at (6.5,2)
	{$2\sA_3+2\sA_1$};
	\node (16) at (9,2)
	{$2\sA_2+4\sA_1$};
	\node (17) at (11,2)
	{\color{red}$8\sA_1$};

\node (6) at (3,-1)
	{$2\sA_3$};

\node (7) at (6.3,-1)
	{$2\sA_2+2\sA_1$};
\node (9) at (9.5,-1)
	{{\color{red}{$6\sA_1$}}};

\node (4) at (8,-2.5)
	{$4\sA_1$};
 \node (3) at (4.7,-2.5)
	{$2\sA_2$};

\node (1) at (6.3,-3.5)
	{$2\sA_1$};

	\path[-]

(3) edge (1)

(6) edge (3)
(6) edge (4)
(7) edge (3)
(7) edge (4)

(12) edge (6)
(12) edge (9)
(14) edge (6)
(15) edge (6)
(15) edge (7)
(15) edge (9)
(16) edge (9)

(18) edge (7)

(22) edge (12)
(22) edge (15)
(22) edge (17)
(23) edge[bend right=10](9)
(24) edge (15)
(23) edge (14)
(23) edge (18)
(24) edge (16)
(24) edge (17)
(1) edge (4)
(7) edge (16)
(4) edge (9)
(9) edge (17)
(26) edge (17)
;
\end{tikzpicture}
\]
}

\begin{rema}
\label{rema:contr}
In contrast to the equivariant case treated in \cite{CTZ}, specialization arguments do not work over $\bR$, and we do not know whether or not a very general cubic threefold with trivial cohomology specializing to one with nontrivial cohomology is rational over $\bR$.
They can be applied to show the existence of cubics
failing stable rationality 
over fields such as $\bR(t)$. 
\end{rema}

\section{Two Singular Points}
\label{sect:two}

Assume that $X_{\bC}$ has two complex-conjugated singular points, e.g., 
$$
p_1=[1:i:0:0:0],\quad\text{and}\quad p_2=[1:-i:0:0:0],
$$
so that $X\subset\bP^4_{x_1,\ldots,x_5}$ is given by 
\begin{align}\label{eqn:2sineq}
(x_1^2+x_2^2)x_3+x_1q_1+x_2q_2+f_3=0,
\end{align}
where $q_1,q_2\in \bR[x_4,x_5]$ are quadratic forms and 
$f_3\in \bR[x_3, x_4, x_5]$ is a cubic form. One  checks that $q_1$ and $q_2$ cannot both vanish if $X_{\bC}$ has $\sA_n$-singularities. If $q_1\not\equiv 0$ and $q_2\equiv 0$, then up to a change of variables,
we have $q_1=q_2$. Thus, we may assume that none of $q_1$ or $q_2$ is $0.$

Each of $q_1$ and $q_2$ defines two points in $\bP^1_{x_4,x_5}(\bC)$, with one of the following possibilities: 
\begin{itemize}
\item one real point with multiplicity $2$, or
\item two real points, or 
\item two conjugate points. 
\end{itemize}
Combining the configuration patterns of four points defined by $q_1$ and $q_2$, we find all possibilities of the pair $(q_1, q_2)$, up to isomorphism:
{\small
$$
\begin{tabular}{|c|c|c|c|c|c|}
\hline
     $q_1$& \multicolumn{5}{c|}{$q_2$}\\\hline
      $x_4^2$& $\lambda x_4^2$&$x_5^2$&$x_4x_5$&$\lambda(x_4^2-x_5^2)$&$\lambda(x_4^2+x_5^2)$\\\hline
$ x_4x_5$&$\lambda x_4x_5$&$\lambda x_4(x_4-x_5)$&$\lambda(x_4^2-x_5^2)$&$\lambda(x_4^2+x_5^2)$&\\\hline
     $x_4^2+x_5^2$&$x_4^2+\lambda x_5^2$&\multicolumn{4}{c|}{$\lambda>0$ everywhere}\\\hline
\end{tabular}
$$
}

Introducing new coordinates 
$$
y_1=x_1x_3+\frac{q_1}{2},\quad \text{and}\quad y_2=x_2x_3+\frac{q_2}{2},
$$
and multiplying \eqref{eqn:2sineq} by $x_3$, we birationally transform $X$ over $\bR$ to 
\begin{align}\label{eqn:conicbundle2}
   Y= \{y_1^2+y_2^2+f_3x_3-\frac{q_1^2+q_2^2}{4}=0\}\subset\bP_{y_1, y_2, x_1,x_2,x_3} (2,2,1,1,1).
\end{align}
Note that $Y$ is a conic bundle over $\bP^2$, with 
$$
\pi: Y\to \bP^2_{x_1,x_2,x_3},\quad (y_1,y_2,x_1,x_2,x_3)\mapsto(x_1,x_2,x_3).
$$
The image $\pi(Y(\bR))$ is given by 
\begin{align}\label{eqn:conicim}
    \frac{q_1^2+q_2^2}{4}-f_3x_3\geq 0.
\end{align}
In particular, we write
\begin{multline}\label{eqn:f3eqn}
f_3:=t_1x_3^3+x_3^2(t_2x_4+t_3x_5)+x_3(t_4x_4^2+t_5x_5^2+t_6x_4x_5)+\\
+t_7x_4^2x_5+t_8x_5^2x_4+t_9x_4^3+t_{10}x_5^3,
\end{multline}
   for  $t_1,\ldots,t_{10}\in\bR.$ 

   \subsection*{2$\sA_1$ and 2$\sA_2$-singularities}
   \begin{prop}\label{prop:2a12a2}
    Assume that $X$ has no real singular points and that $X_{\bC}$ has $2\sA_1$ or $2\sA_2$-singularities. Then $X$ is not rational over $\bR.$
   \end{prop}
   
   \begin{proof}
   Recall that $X$ is birational to the conic bundle \eqref{eqn:conicbundle2}. When $X_\bC$ has $2\sA_1$ or $2\sA_2$-singularities, the conic bundle has a smooth quartic discriminant, with trivial double cover. By \cite{BW} or \cite[Theorem 6.10]{KuzPro}, $X$ is not rational over $\bR.$
   \end{proof}
   

\subsection*{$2\sA_3$-singularities with no plane}
When $p_1$ and $p_2$ are $\sA_3$-points and $X_{\bC}$ contains no plane, one can check (cf. \cite[Section 5]{CMTZ}) that the only possibility for $q_1$ and $q_2$ from the table above is
$$
q_1=x_4x_5\quad \text{ and }\quad q_2=\frac{1}{2}(x_4^2-x_5^2).
$$
The $\sA_3$-singularities impose the following conditions on \eqref{eqn:f3eqn} 
$$
 t_7=t_{10},\quad t_8=t_9.
 $$

\begin{rema}
   Consider the conic bundle $\pi: Y\to\bP^2$ birational to $X$ given by \eqref{eqn:conicbundle2}. Note that the real locus of the discriminant curve in $\bP^2$ is smooth. It follows that $Y(\bR)$ (and thus $X(\bR)$) is connected if and only if the image $\pi(Y(\bR))$ in $\bP^2(\bR)$ given by \eqref{eqn:conicbundle2} is connected.
   Using this, we present examples where $X(\bR)$ is disconnected or connected:
   \begin{itemize}
       \item   Disconnected: $
        t_1=-1, t_2=-10, t_3=-2, t_4=t_5=t_6=t_7=1, t_8=5,
   $
   \item Connected: $
        t_1=t_2=t_3=t_4=t_5=t_6=t_7=t_8=1.
   $
   \end{itemize}
  
\end{rema}

\subsection*{$2\sA_3$-singularities with one plane}

Here, $X_{\bC}$ contains a unique plane $\Pi$, which is thus defined over $\bR$. This implies that $X$ is birational to a smooth intersection of two quadrics $X_{2,2}\subset \bP^5$, via unprojection from the plane, see \cite[Proposition 5]{CTZ}. Rationality of $X_{2,2}$ over $\bR$ is determined by the existence of a real line, see \cite{HT-quad}; on $X$ this translates into the existence of a real line disjoint from $\Pi$.

\begin{exam}
The cubic $X$ given by
    $$
(x_1^2+x_2^2)x_3+(x_1+x_2)x_4^2+(x_3+x_4+x_5)(x_3^2+x_4x_5)-2x_3x_5^2=0 
    $$
    is rational.
    Indeed, the line 
    $$
    \{x_1+x_2=x_1-x_5=x_3+x_4+x_5=0\}
    $$
    in $X$ is disjoint from the unique plane 
    $
    \{x_3=x_4=0\}\subset X
    $.
\end{exam}
Topological types of $X_{2,2}$ with real points but without real lines are listed in \cite[Section 11.4]{HT-quad}.

\subsection*{$2\sA_4$-singularities}
Cubic threefolds $X_{\bC}$ with $2\sA_4$-singularities do not contain a plane. As in the case of $2\sA_3$-singularities with no plane, we find that 
$$
q_1=x_4x_5,\quad q_2=\frac{1}{2}(x_4^2-x_5^2), \quad    t_7=t_{10},\quad t_8=t_9,
$$
together with the following conditions ensuring $2\sA_4$-singularities
$$
t_6=-8t_7t_8,\quad
t_4=t_5 + 4t_7^2 - 4t_8^2.
$$

\begin{exam}
  Similarly as before, $X(\bR)$ can be connected or disconnected depending on the parameters. We give examples: 
\begin{itemize}
    \item Disconnected: $
t_1=2, t_2=-10, t_3=t_5=t_6=t_7=0, t_4=-16, t_8=-2.
$
\item Connected: $
t_1=t_2=t_3=t_4=t_5=t_7=t_8=1,\quad t_6=-8.
$
\end{itemize}
Again, this can be seen via checking the connectedness of \eqref{eqn:conicim}.
\end{exam}

\subsection*{$2\sA_5$-singularities} 
Such cubics $X$
are given by 
\begin{equation} \label{eqn:2a5eqn2}
(x_1^2 + x_2^2)x_3 + x_3^3 +\frac12x_2(x_4^2-x_5^2) + x_1x_4x_5 + bx_3(x_4^2+x_5^2)=0,\quad b\in\bR.
\end{equation}
Note that $X(\bR)$ is connected, see Example~\ref{exam:2A5-coho}. However $X$ is not stably rational:
\begin{prop}\label{prop:2a5h1}
Assume that $X$ has no real singular points and that $X_{\bC}$ has $2\sA_5$-singularities.
Then $X$ is not stably rational.
\end{prop}

\begin{proof}
We may assume that $X$ is given by \eqref{eqn:2a5eqn2}. The equation 
    $$
    q(x_1,\ldots,x_5)=x_3^2+\frac{2b-\sqrt{4b^2+2}}{4}x_4^2+\frac{2b-\sqrt{4b^2+2}}{4}x_5^2
    $$
    cuts out complex conjugated divisors on $X_{\bC}$, we have 
    $$
    \{q=0\}\cap X=S\cup\bar S,
    $$
    where
    \begin{multline*}
        S=\{q=x_3(x_1+ix_2)+\frac{i}{4}(x_4-ix_5)^2=x_1x_4+ix_2x_4-x_2x_5+\\+ix_1x_5+(\sqrt{4b^2+2}+2b)ix_3x_4+(\sqrt{4b^2+2}+2b)x_3x_5=0\}.
    \end{multline*}
    By \cite[Section 5]{CMTZ}, the class group $\Cl(X_\bC)$ is generated by the classes of $S, \bar S$ and the class of a hyperplane section $F$ subject to the relation $S+\bar S=2F$. As in \cite[Proposition 5.12]{CMTZ}, 
    $$
    \rH^1(\Gamma,\Pic(\tilde X_\bC))=\bZ/2,
    $$
    hence $X$ is not stably rational over $\bR$.
\end{proof}

Here we also record an explicit specialization pattern. See Remark~\ref{rema:contr} for a discussion of specializations over $\bR$.
\begin{exam}\label{exam:2A3spec}
 Consider the family of cubics $\cX\to\bA^{8}_{t_1,\ldots,t_8}$ given by 
    \begin{multline*}         (x_1^2+x_2^2)x_3+x_1x_4x_5+\frac12x_2(x_4^2-x_5^2)+t_1x_3^3+x_3^2(t_2x_4+t_3x_5)+\\+x_3(t_4x_4^2+t_5x_5^2+t_6x_4x_5)
+(x_4^2+x_5^2)(t_7x_5+t_8x_4)=0.
    \end{multline*}
    Let $\cX_1$ be the subfamily of $\cX$ consisting of fibers above the locus
    $$
    \{t_6+8t_7t_8=t_5-t_4+4t_7^2-4t_8^2=0\}\subset\bA^8,
    $$
 and $\cX_2$ be the subfamily of $\cX$ consisting of fibers above the line 
    $$
    \{t_1-1=t_2=t_3=t_5-t_4=t_6=t_7=t_8=0\}\subset\bA^8.
    $$
    We have a natural inclusion $\cX\supset\cX_1\supset\cX_2$. Then a very general fiber of $\cX$ is a real cubic with  $2\sA_3$-singularities, that of $\cX_1$ has $2\sA_4$-singularities, and that of $\cX_2$ has $2\sA_5$-singularities, over $\bC$.
\end{exam}

\subsection*{$2\sD_4$-singularities}
As in the proof of \cite[Proposition 5.15]{CMTZ}, we find that such $X$ are given by 
\begin{multline} \label{eqn:2d4eq}
(x_1^2 + x_2^2)x_3 +t_1x_3^3+x_3^2(t_2x_4+t_3x_5)+\\
+x_3(t_4x_4^2+t_5x_5^2+t_6x_4x_5)
+x_4q(x_4,x_5)=0,
\end{multline}
for general parameters $t_1,\ldots,t_6\in\bR$, and 
$$
q(x_4,x_5)=x_4^2-x_5^2\quad\text{ or } \quad x_4^2+x_5^2.
$$
In the first case, $X$ contains three real planes
$$
\Pi_1=\{x_3=x_4=0\},  \Pi_2=\{x_3=x_4+x_5=0\}, \Pi_3=\{x_3=x_4-x_5=0\}.
$$
In the second case, $X$ contains only one real plane $\Pi_1$. In both cases, projection from $\Pi_1$ induces a 
quadric surface bundle
$$
\pi:X\dashrightarrow \bP^1_{x_3,x_4}.
$$
Note that the discriminant curve of the conic bundle structure \eqref{eqn:conicbundle2} in this case has real singularities. It is therefore more appropriate to use the quadric surface bundle structure to study the topology of $X(\bR).$ We have 
$$
\# \text{ components  of } X(\bR) \,=\, \# \text{ components  of } \pi(X(\bR)).
$$

First, we consider the case $q=x_4^2-x_5^2$. 
There are three possibilities:
\begin{enumerate} 
    \item $X$ contains a line disjoint from a plane, then $X$ is rational, 
    \item 
    $X(\bR)$ has two components, and $X$ is not stably rational, 
    \item $X(\bR)$ is connected but $X$ does not contain any line disjoint from one of the planes $\Pi_1$, $\Pi_2$, $\Pi_3$. 
\end{enumerate}
In the third case, we do not know whether or not $X$ is rational. Each possibility is realized, e.g., by
  \begin{enumerate}
      \item $(x_1^2+x_2^2)x_3-x_3^3+x_4(x_4^2-x_5^2)$,
      \item $(x_1^2+x_2^2)x_3+300x_3^3+x_3^2(35x_4+24\sqrt{21}x_5)+10x_3x_5^2+x_4(x_4^2-x_5^2)$,
      \item $(x_1^2+x_2^2)x_3+x_3^3+x_4(x_4^2-x_5^2)$.
  \end{enumerate}
Let us explain when these possibilities occur, using $\pi$. Alternatively, we can use projection from $\Pi_2$ or $\Pi_3$, but this is essentially the same, up to a change of coordinates. Put 
\begin{multline*}
    \Delta(x_3,x_4):=-x_4^4+ ( t_5-t_4)x_3x_4^3 + ( t_4t_5 -t_2- \frac14t_6^2)x_3^2x_4^2+\\+ ( t_2t_5-t_1 - \frac12t_3t_6)x_3^3x_4 +(t_1t_5 - \frac14t_3^2)x_3^4.
\end{multline*}
Then $\Delta(x_3,x_4)\cdot x_3^2$ is the discriminant of $\pi$. Set $D_1(x)=\Delta(1,x)$. Note that $D_1(x)$ can only have simple real roots since $X$ has only three planes and no real singular points. Every line in $X$ intersects exactly one plane among $\Pi_1, \Pi_2,\Pi_3$,
since $\Pi_1+\Pi_2+\Pi_3$ is cut out on $X$ by the hyperplane $x_3=0$.

\begin{lemm}
The cubic $X$ contains a line disjoint from $\Pi_2$ and $\Pi_3$ if and only if  
$t_3=-t_5t_6$ and $t_1+t_2t_5+t_4t_5^2+t_5^3<0$, or
$D_1$ has real roots and $t_5$ is smaller than its largest root.
\end{lemm}

\begin{proof}
Observe that $X$ contains a line disjoint from $\Pi_2$ and $\Pi_3$ if and only if this line is contained in a fiber of $\pi$ over $(x_3,x_4)\ne(0,1)\in \bP^1$.
Thus, we may set $x_3=1$.
We have to understand when the real locus of the quadric $\pi^{-1}(x_4)$ contains a real line. If $x_4=t_5$, the real locus of $\pi^{-1}(t_5)$ is isomorphic to the quadric in $\bP^3_{u_1,u_2,u_3,u_4}$ given by 
$$
u_1^2+u_2^2+(t_1 + t_2t_5 + t_4t_5^2 + t_5^3)u_3^2-\frac{(t_3+t_5t_6)^2}{4(t_1 + t_2t_5 + t_4t_5^2 + t_5^3)}u_4^2=0.
$$
If $t_3=-t_5t_6$ and $t_1 + t_2t_5 + t_4t_5^2 + t_5^3<0$, the real locus of this quadric is a cone, so the fiber $\pi^{-1}(t_5)$ contains a line that is disjoint from $\Pi_2$ and $\Pi_3$.
Similarly, if $t_3=-t_5t_6$ and $t_1 + t_2t_5 + t_4t_5^2 + t_5^3>0$,
the real locus of this quadric is just a point.
In all other cases, the real locus is a sphere, so it does not contain lines either.

When $x_4\ne t_5$, the fiber $\pi^{-1}(t_5)$ is isomorphic to the quadric surface 
\begin{align}\label{eqn:diagqua2d4}
    S=\{u_1^2+u_2^2+(t_5-x_4)u_3^2+\frac{D_1(x_4)}{t_5-x_4}u_4^2=0\}\subset \bP^3_{u_1,u_2,u_3,u_4},
\end{align}
and we may have the following possibilities:
\begin{itemize}
    \item $D_1(x_4)<0$, $S(\bR)$ a sphere;
    \item $D_1(x_4)>0$ and $x_4>t_5$, $S(\bR)$ is a hyperboloid, so the fiber contains a line disjoint from $\Pi_2$ and $\Pi_3$;
     \item $D_1(x_4)> 0$ and $x_4<t_5$, $S(\bR)$ is empty; 
      \item $D_1(x_4)= 0$ and $x_4<t_5$, $S(\bR)$ is a point; 
   \item $D_1(x_4)=0$ and $x_4>t_5$, $S(\bR)$ is a quadric cone, so the fiber contains a line disjoint from $\Pi_2$ and $\Pi_3$.
\end{itemize}
This implies the required assertion.
\end{proof}

Similarly, we obtain:

\begin{lemm}
The locus $X(\bR)$ has two connected components if and only if  $D_1$ has 4 real roots and $t_5$ is greater than these roots.
\end{lemm}

\  

Now, we consider the case $q=x_4^2+x_5^2$. 
In this case $X$ contains only one real plane, 
$X$ does not contain real lines disjoint this plane,
and the fiber of $\pi$ over $x_3=0$ is empty. As before, we have two possibilities: 
\begin{enumerate}
    \item $X(\bR)$ is disconnected, so $X$ is not stably rational over $\bR$,
    \item $X(\bR)$ is connected.
\end{enumerate}
In the second case, we do not know whether or not $X$ is rational. Both possibilities are realized:
\begin{enumerate}
    \item $X=\{(x_1^2+x_2^2)x_3+x_3^3+x_4(x_4^2+x_5^2)=0\}$,
    \item $X=\{(x_1^2+x_2^2)x_3+x_3^3-x_3^2x_4+2x_3(x_5^2-x_4^2)+x_4(x_4^2+x_5^2)=0\}$.
\end{enumerate}
To explain these cases in terms of the defining equation \eqref{eqn:2d4eq}, set 
\begin{multline*}
   D_2(x):=x^4 + (t_4 + t_5)x^3 + (t_2 + t_4t_5 - \frac14t_6^2)x^2+\\ + (t_1 + t_2t_5 - \frac12t_3t_6)x + t_1t_5 - \frac14t_3^2.
\end{multline*}
We have:

\begin{lemm}\label{lemm:topo2d41pl}
The locus $X(\bR)$ is connected if and only if one of the following holds:
\begin{itemize}
 \item $D_2$ has no real root, 
  \item $D_2$ has 2 real roots $\alpha_1<\alpha_2$ and $\alpha_1\leq-t_5$,
   \item $D_2$ has 4 real roots $\alpha_1<\alpha_2<\alpha_3<\alpha_4$ and $\alpha_3\leq-t_5$.
\end{itemize}

\end{lemm}

\begin{proof}
  Recall that $X(\bR)$ is connected if and only if $\pi(X(\bR))$ is. Set $x_3=1$. When $x_4=-t_5$, the real locus of $\pi^{-1}(x_4)$ is isomorphic to the quadric in $\bP^3_{u_1,u_2,u_3,u_4}$ given by 
 $$
u_1^2+u_2^2+(t_1 - t_2t_5 + t_4t_5^2 - t_5^3)u_3^2-\frac{(t_3-t_5t_6)^2}{4(t_1 - t_2t_5 + t_4t_5^2 - t_5^3)}u_4^2=0.
$$
One can check that the real locus of this quadric is always nonempty. 
When $x_4\ne-t_5$, a general fiber of $\pi$ is isomorphic to the quadric surface in $\bP^3_{u_1,u_2,u_3,u_4}$
$$
S=\{u_1^2+u_2^2+(x_4+t_5)u_3^2+\frac{D_2(x_4)}{x_4+t_5}u_4^2=0\}
$$
and we may have the following possibilities:

\begin{itemize}
    \item when $D_2(x_4)<0$, $S(\bR)$ is a sphere,
    \item when $D_2(x_4)>0$ and $x_4<-t_5$, $S(\bR)$ is a hyperboloid, 
     \item when $D_2(x_4)> 0$ and $x_4>-t_5$, $S(\bR)$ is empty, 
   \item when $D_2(x_4)=0$ and $x_4<-t_5$, $S(\bR)$ is a quadric cone,
      \item when $D_2(x_4)=0$ and $x_4>-t_5$, $S(\bR)$ is a point.
\end{itemize}
From this, one obtains the cases when $X(\bR)$ is connected, based on roots of $D_2$, as is stated. 
\end{proof}



\section{Four singular points}

When four singular points are not in general position, by \cite{CTZ}, all of them are $4\sA_1$-singularities, contained in a plane $\Pi\subset X$.

\begin{prop}
    Assume that $X$ has no real singular points and that $X_\bC$ has $4\sA_1$-singularities in a plane $\Pi\subset X$ . Then $X$ is rational if and only if $X$ contains a  real line disjoint from $\Pi.$
\end{prop}
\begin{proof}
    Unprojecting from $\Pi$, $X$ is birational to a smooth intersection of two quadrics  $X_{2,2}$ in $\bP^5$, 
    which is rational if and only if it contains a real line, by \cite{HT-quad}. The existence of a real line in $X_{2,2}$ is equivalent to the existence of a real line in $X$ disjoint from $\Pi.$ 
\end{proof}

\begin{exam}
The cubic $X$ given by
    $$
    (x_1^2+x_2^2+x_3^2)x_4+
    x_2(x_3x_5+x_4^2)-x_5(x_4^2+3x_4x_5+x_5^2)=0
    $$
    is rational. Indeed, 
    the line 
    $$
    \{x_1-x_2=x_1-x_3=x_1-x_5=0\}
    $$
    in $X$ is disjoint from the unique plane 
    $
    \{x_4=x_5=0\}\subset X
    $.
\end{exam}

For the rest of the section, we assume that the singular points of $X_{\bC}$ are in general position and have coordinates
$$
p_1=[1:i:0:0:0], \quad p_2=[1:-i:0:0:0],
$$
$$
p_3=[0:0:1:i:0],\quad p_4=[0:0:1:-i:0].
$$
Then the defining equation of $X$ is 
\begin{multline}\label{eqn:4singnormal}
    (r_1x_1+r_2x_2)(x_3^2+x_4^2)+ (r_3x_3+r_4x_4)(x_1^2+x_2^2)+\\+ax_5^3+x_5^2(b_1x_1+b_2x_2+b_3x_3+b_4x_4)+\\+
    x_5(t_1x_1x_3+t_2x_1x_4+t_3x_2x_3+t_4x_2x_4+t_5(x_1^2+x_2^2)+t_6(x_3^2+x_4^2)),
\end{multline}
with parameters $r_1,\ldots,r_4,b_1,\ldots,b_4,t_1,\ldots,t_6,a\in\bR.$

\subsection*{$4\sA_1$, $2\sA_2+2\sA_1$, $4\sA_2$-singularities} 
In these cases, at least one of the $r_1$ and $r_2$, and one of the $r_3$ and $r_4$ is nonzero. After changing variables, 
$$
r_1=r_2=r_3=r_4=1, \quad t_1=t_4, \quad t_2=t_3,\quad t_5=t_6=0.
$$
It follows that $X$ is given by
\begin{multline}\label{eqn:4a1normal}
    (x_1+x_2)(x_3^2+x_4^2)+(x_3+x_4)(x_1^2+x_2^2)+ax_5^3+\\+x_5^2(b_1x_1+b_2x_2+b_3x_3+b_4x_4)+\\+
x_5(t_1(x_1x_3+x_2x_4)+t_2(x_1x_4+x_2x_3))=0,
\end{multline}
with parameters $a,b_1,b_2,b_3,b_4,t_1,t_2\in\bR.$ The singularity types of $X$ are determined by the parameters as follows
\begin{enumerate}
    \item $4\sA_1$: general $a,b_1,b_2,b_3,b_4,t_1,t_2\in\bR$,
    \item $2\sA_2+2\sA_1$: $b_1=b_2=-\frac{(t_1-t_2)^2}{8}$,
    \item $4\sA_2$: $b_1=b_2=b_3=b_4-\frac{(t_1-t_2)^2}{8}$.
\end{enumerate}
Cubics in the second and third cases are in fact birational to cubics with 
$4\sA_1$-singularities. 
Common birational models for all three types are forms of smooth divisors in $(\bP^1)^4$ of degree $(1,1,1,1)$, see \cite[Section 5]{CTZ}. We describe these models over $\bR$.

Let $\pi\colon\widetilde{X}\to X$ be the blowup of the singularities of $X$.
Then there exists a commutative diagram:
$$
\xymatrix{
\widetilde{X}\ar@{-->}[rr]^{\rho}\ar@{->}[d]_{\pi}&&\widehat{X}\ar@{->}[d]^{\phi}\\%
X\ar@{-->}[rr]^{\chi}&&Y}
$$
where 
\begin{itemize} 
\item 
$\rho$ is a composition of flops in the strict transform of the lines passing 
through pairs of singular points, 
\item $\phi$ is a contraction of the strict transform of 
the hyperplane section containing $p_1\ldots,p_4$ to a smooth point $q$ of the threefold $Y$,
and
\item 
$Y$ is a form of a smooth divisor in $(\mathbb{P}^1)^4$ of degree $(1,1,1,1)$ with Picard rank 2. 
Such divisors with Picard rank 1 are conjecturally irrational, see \cite[Conjecture 1.3]{KuznetsovProkhorov2022}. 
\end{itemize}

Over $\bR$, $Y$ can be embedded in $\bP^3_{y_1,\ldots y_4}\times\bP^3_{z_1,\ldots z_4}$ with equations 
$$
\begin{cases}
    y_1^2=y_2^2+y_3^2+y_4^2,\\
    z_1^2=z_2^2+z_3^2+z_4^2,\\
    \mathbf{y}\cdot M\cdot\mathbf{z^t}=0,
\end{cases}
$$
where 
$
\mathbf{y}=(y_1,\ldots,y_4), 
\mathbf{z}=(z_1,\ldots,z_4),
$
and $M$ is a $4\times 4$ matrix with real entries. Finding simple normal forms of $Y$ is related to the singular value decomposition in Minkowski space \cite{svdminkowski}.

The map $\chi^{-1}$ is given by the linear system $|\phi^*(-K_Y)-3E)|$, where $E$ is the exceptional divisor of $\phi$ above $q.$ Different choices of $q$ on $Y$ lead to birational maps to different cubic threefolds with 4 singular points. This allows to find birational maps between cubics with $4\sA_1, 4\sA_2$ and $4\sA_2+2\sA_1$-singularities. 
Explicitly, let $X$ be a cubic given by \eqref{eqn:4a1normal} with $2\sA_2+2\sA_1$ or $4\sA_2$-singularities. Let $p_5$ be a general point on $X$. Then the 
restriction to $X$ of the linear subsystem in $|\cO_{\bP^4}(4)|$ consisting of quartics
having singularities of multiplicity 3 at the points $p_1,\ldots, p_5$ 
gives a birational map $X\dashrightarrow X'$, where $X'$ is a cubic threefold with $4\sA_1$-singularities in general position. 

We study the topology of $X(\bR)$. As in \eqref{eqn:conicbundle2}, $X$ is birational to the conic bundle 
$$
Y=\{y_1^2+y_2^2+(x_3+x_4)f_3-\frac{q_1^2+q_2^2}{4}\}\subset\bP_{y_1,y_2,x_1,x_2,x_3}(2,2,1,1,1),
$$
with 
\begin{align}\label{eqn:4sinconic}
\pi: Y\to\bP^2_{x_3,x_4,x_5},\quad (y_1,y_2,x_3,x_4,x_5)\mapsto (x_3,x_4,x_5),
\end{align}
where 
\begin{align*}
q_1&=x_3^2+x_4^2+b_1x_5^2+t_1x_3x_5+t_2x_4x_5,\\
q_2&=x_3^2+x_4^2+b_2x_5^2+t_1x_4x_5+t_2x_3x_5,\\
f_3&=ax_5^3+b_3x_3x_5^2+b_4x_4x_5^2.
\end{align*}
The real locus of the discriminant curve in $\bP^2$ is smooth for all three singularity types of $X$. It follows that $Y(\bR)$ and $X(\bR)$ are connected if and only if $\pi(Y(\bR))$ is. 
\begin{exam} 
We provide examples with $X(\bR)$
    \begin{itemize}
    \item disconnected: $
a=-3, b_1=b_2=b_3=b_4=0, t_1=3, t_2=4,
$\item connected: $
a=0, b_1=b_2=b_3=b_4=t_1=t_2=1.
$    \end{itemize}   

\end{exam}


\subsection*{$2\sA_3+2\sA_1$-singularities}
Assume that $p_1$, $p_2$ are $\sA_3$-singularities and $p_3$, $p_4$ are $\sA_1$-singularities. From \cite[Section 7]{CMTZ}, we know that the defect of $X_\bC$ is 1 or 2. 
When the defect is $1,$ by \cite[Lemma 7.3]{CMTZ}, $X$ is rational over $\bR$. 
When the defect is $2$, $X_\bC$ contains three planes. Two of them are spanned by
$$
\Pi_1=\langle p_2,p_3,p_4\rangle,\quad\Pi_2=\langle p_1,p_3,p_4\rangle,
$$
and thus are complex conjugate. 
This implies that $r_1=r_2=0$ in \eqref{eqn:4singnormal}. The normal form of such $X$ is 
\begin{multline*}
    x_4(x_1^2+x_2^2)+ax_5^3+x_5^2(b_1x_1+b_2x_2+b_3x_3+b_4x_4)+\\+
    x_5(t_2(x_1+x_2)x_4+t_6(x_3^2+x_4^2)),
\end{multline*}
with parameters $b_1,b_2,b_3,b_4,t_2,t_6,a\in\bR.$ The third plane is defined over $\bR$ and is given by
$$
\Pi_3=\{x_4=x_5=0\}.
$$
As before, projecting from $\Pi_3,$ the projection map 
\begin{align*}
    \pi: X\to \bP^1, \quad (\mathbf x)\mapsto (x_4, x_5),
\end{align*}
endows $X$ with the structure of a quadric surface bundle.

\begin{prop}
Let $\Delta\in\bR[x_5]$ be given by  
$$
\Delta(x_5):=-\frac{b_1^2+b_2^2}{4} x_5^3 + \frac{-2t_2t_6(b_1+b_2) -b_3^2 + 4at_6}{4t_6}x_5^2+ (b_4 
    - \frac12t_2^2)x_5 + t_6.
$$
The locus $X(\bR)$ is disconnected if and only if $\Delta(x_5)$ has three real roots $\alpha_1<\alpha_2<\alpha_3$ and one of the following holds
\begin{itemize}
 \item $t_6>0$ and $0<\alpha_1$, or
  \item $t_6<0$ and $\alpha_3<0$.
\end{itemize}

\end{prop}
\begin{proof}
Set $x_4=1$. The fiber $\pi^{-1}(x_5)$ is isomorphic to the quadric surface 
$$
\{u_1^2+u_2^2+t_6x_5u_3^2+x_5\Delta(x_5)u_4^2 =0\}\subset\bP^3_{u_1,\ldots,u_4}
$$
Its real locus is empty if and only if 
$
t_6, x_5
$ and $\Delta(x_5)$ have the same sign. Note that $t_6\ne 0$ due to the singularities of $X_\bC$. 
\end{proof}

\subsection*{$2\sD_4+2\sA_1$-singularities}Assume $p_1$ and $p_2$ are $\sD_4$-singularities and $p_3$ and $p_4$ are $\sA_1$-singularities of $X_\bC.$ 
We recall some basic geometry from \cite[Section 7]{CMTZ}. The class group $\Cl(X_\bC)$ is generated by the classes of five planes $\Pi_1,\ldots,\Pi_5$ in $X_\bC$ and the class of a general hyperplane section $F$, subject to relations 
$$
\Pi_1+\Pi_2+\Pi_4=\Pi_3+\Pi_4+\Pi_5=F\in\Cl(X_\bC).
$$
One observes that $\Pi_4$ is defined over $\bR$. We have  
 \begin{align}\label{eqn:2d42a1plpt}
     \Pi_1\supset\{p_2, p_3, p_4\},\quad \Pi_2\supset\{p_1, p_3, p_4\},\quad \Pi_3,\Pi_4,\Pi_5\supset\{p_1,p_2\}.
 \end{align}
The line passing through $p_3$ and $p_4$ is disjoint from $\Pi_3$ and $\Pi_5$.

Now consider the equation \eqref{eqn:4singnormal}. The inclusions in \eqref{eqn:4a1normal} imply that $r_1=r_2=0.$ Up to isomorphism, we may assume 
$$
r_4=t_1=t_4=t_5=0,\quad r_3=1.
$$
The $\sD_4$-singularities impose the conditions 
$$
t_2+it_3=-it_3t_5-b_1+ib_2=0.
$$
Since all parameters are real numbers, we have 
$$
t_2=t_3=b_1=b_2=0.
$$
After a change of variables, the normal form of $X$ is given by
\begin{align}\label{eqn:2d4+2a1normal}
    (x_1^2 + x_2^2)x_3 + x_5(x_3^2 +x_4^2) + (b_3x_3+b_4x_4)x_5^2 + ax_5^3=0,
\end{align}
for $a\in\bR^\times$ and $b_3,b_4\in\bR.$ Note that a further change of variables can set $a=\pm1.$
Over $\bC$, the 5 planes are given by
$$
\Pi_1=\{x_1+ix_2=x_5=0\},\quad \Pi_2=\{x_1-ix_2=x_5=0\},
$$
$$
\Pi_3=\{x_3=x_5=0\},\quad\Pi_4=\{x_3=2x_4+\left(b_4+\sqrt{b_4^2-4a}\right)x_5=0\},
$$
$$
\Pi_5=\{x_3=2x_4+\left(b_4-\sqrt{b_4^2-4a}\right)x_5=0\}.
$$
 
 Rationality of $X$ is determined by the {\bf (H1)}-obstruction. The following is analogous to \cite[Corollary 7.8]{CMTZ}.

\begin{prop}
\label{prop:2d4+2a1H1}
Assume that $X$ has no real singular points
and that $X_{\bC}$ has $2\sD_4+2\sA_1$ singularities. 
Then $X$ is not (stably) rational over $\bR$ if and only if $X$ contains only one real plane.
\end{prop}

\begin{proof}
From the equations we see that either $X$  contains only one plane $\Pi_4$ or three planes $\Pi_3,\Pi_4,\Pi_5$. Assume $X$ contains only $\Pi_4$. Then the Galois group $\Gamma$ switches $\Pi_1\leftrightarrow\Pi_2$ and  $\Pi_3\leftrightarrow\Pi_5$. One computes, cf. \cite[Proposition 7.7]{CMTZ}, that 
   \begin{align}\label{eq:2d42a1h1}
       \rH^1(\Gamma,\Pic(\tilde X_\bC))=\bZ/2,
   \end{align} 
   and thus $X$ is not stably rational over $\bR.$

   When $X$ contains three planes, the real line passing through $p_3$ and $p_4$ is disjoint from $\Pi_3$, thus $X$ is rational over $\bR.$
\end{proof}
\begin{coro}
    Let $X$ be given by \eqref{eqn:2d4+2a1normal}. Then $X$ is (stably) rational if and only if $b_4^2>4a$.
\end{coro}
\begin{rema}
The quadric surface bundle structure $\pi: X\to \bP^1$ obtained via projection from $\Pi_4$ has been used, in \cite[Section 11.2]{CTA},
to establish the nontriviality of cohomology in \eqref{eq:2d42a1h1}, via a computation of the Brauer group and the identity
$$
\rH^1(\Gamma,\Pic(\tilde X_\bC))=\Br(\tilde X_\bR)/\Br(\bR).
$$
    
\end{rema}
\section{Six singular points}
\label{sect:six}
\subsection*{$6\sA_1$-singularities with no plane}
\begin{prop}
    Let $X$ be a cubic without real singular points and such that $X_{\bC}$ has $6\sA_1$-singularities in linearly general position. Then $X$ is (stably) rational over $\bR$ 
    if and only if $X$ contains a normal rational cubic scroll defined over $\bR$. 
\end{prop}

\begin{proof}
Recall from \cite[Section 7]{CTZ} that the class group of $X_\bC$ is generated by the hyperplane section $H$ and two classes of normal rational cubic scrolls $S_1$ and $S_2$, subject to the relation $2H=S_1+S_2$. 

First we show that $X$ contains a real normal rational cubic scroll if and only if the classes $S_1$ and $S_2$ are $\Gamma$-invariant. 
 Assume that $S_1$ and $S_2$ are $\Gamma$-invariant, then a general section of the linear system $|S_1|$ is such a cubic scroll. The other direction is easy to see.

When $X$ does not contain a real cubic scroll,  $\Gamma$ switches $S_1$ and $S_2$, and one can compute, cf. \cite[Proposition 7.5]{CTZ}, that 
$$
\rH^1(\Gamma,\Pic(\tilde X_\bC))=\bZ/2\bZ.
$$
It follows that $X$ is not stably rational over $\bR$. 

When $X$ contains a real normal rational cubic scroll. The linear system $|\cO_{X}(2)-S|$ gives a rational map $X\dashrightarrow \bP^2$, whose resolution gives $X$ the structure of a $\bP^1$-bundle over $\bP^2$, see \cite[Section 7]{CTZ}. Such a $\bP^1$-bundle over $\bR$ admits a section and is therefore rational.
\end{proof}

\begin{rema}
   In \cite[Proposition 7.1]{CTZ}, it is shown that any involution $\iota\in \Aut(X_\bC)$ not fixing any singular points acts nontrivially on  $\Cl(X_{\bC})$, see also Example~\ref{exam:6A1}.

    However, this does not hold for the Galois 
    action. For example, let $X$ be the cubic given by 
    $$
x_1(x_3^2-x_4^2-x_5^2) + 5x_2x_3x_4 - x_5(x_1^2+4x_2^2-x_3^2) =0.
    $$
    Then $X_{\bC}$ has three pairs of conjugate $\sA_1$-singular points in linearly general position. But $\Gamma$ acts trivially on $\Cl(X_\bC)$. In particular, $X$ contains a real cubic scroll given by
    \begin{multline*}
        \{-4x_2x_3 - x_3^2 + x_1x_4 + x_1x_5
=x_1x_3 + x_2x_4 - x_3x_4 +\\+ x_2x_5 + x_3x_5=
x_1^2 + 4x_2^2 - 3x_2x_3 - x_3^2 + 2x_1x_5=0\}\subset X.
    \end{multline*}
Therefore $X$ is rational. In this case, the Galois action does not fix any singular points but acts trivially on $\Cl(X_\bC)$ and satisfies ${\bf (H1)}$.  This does not happen in the equivariant setting.
    
Recall that there are also examples where the Galois action fails ${\bf (H1)}$ and the corresponding cubic is not stably rational, see Example~\ref{exam:6A1}.
\end{rema}

\subsection*{$6\sA_1$-singularities with one plane}

All such cubics are rational:
\begin{prop}\label{prop:6a1h1}
    Let $X$ be a cubic such that $X_{\bC}$ has $6\sA_1$-singularities and one plane. Then $X$ is rational over $\bR$.
\end{prop}

\begin{proof}
The unique plane is defined over $\bR$ and contains four of the $\sA_1$-points, see \cite[Section 7]{CTZ}. The line through the two other $\sA_1$-points is also defined over $\bR$ and is disjoint from the plane. This implies rationality of $X$ over $\bR$.
\end{proof}

\subsection*{$6\sA_1$-singularities with three planes}

\begin{prop}
Let $X$ be a cubic without real singular points and such that $X_{\bC}$ has 
$6\sA_1$-singularities and three planes. Then up to isomorphism, one of the following holds:
    \begin{itemize}
        \item $X$ contains three real planes, and is given by
         \begin{align}
\label{eqn:6a13pl1}
x_2x_3x_4+ax_1^3+x_1^2(a_1x_2+a_2x_3+a_3x_4)+x_1(x_2^2+x_3^2+x_4^2+x_5^2)=0
    \end{align}
     with $a,a_1,a_2,a_3\in\bR$.
    \item $X$ contains only one real plane, and is given by
\begin{align}\label{eqn:6a13pl2}x_2(x_3^2+x_4^2)+ax_1^3+x_1^2(a_1x_2+a_2x_3+a_3x_4)+x_1(x_2^2-x_4x_5+x_5^2)=0,
    \end{align}
       with $a,a_1,a_2,a_3\in\bR$.
    \end{itemize}
\end{prop}

\begin{proof}
One may assume that the three planes of $X_{\bC}$ are contained in the hyperplane $\{x_5=0\}$. If the planes are defined over $\bR$ then we may assume that they are given by 
$$
\Pi_1=\{x_1=x_2=0\},\quad\Pi_2=\{x_1=x_3=0\},\quad\Pi_3=\{x_1=x_4=0\}
$$
and that the singular points of $X$ are
$$
p_1=[0:0:0:1:i],\quad p_2=[0:0:0:1:-i],\quad p_3=[0:1:0:0:i],
$$
$$
p_4=[0:1:0:0:-i],\quad p_5=[0:0:1:0:i],\quad p_6=[0:0:1:0:-i].
$$
All cubic threefolds singular at $p_1,\ldots, p_6$ are given by 
\begin{multline*}
x_2x_3x_4+ax_1^3+x_1^2(a_1x_2+a_2x_3+a_3x_4+a_4x_5)+\\+bx_1(x_2^2+x_3^2+x_4^2+x_5^5)+x_1(c_1x_2x_3+c_2x_2x_4+c_3x_3x_4)=0
\end{multline*}
with $a,b,a_1,a_2,a_3,a_4,c_1,c_2,c_3\in\bR$.
Up to a change of variables,
$$
a_4=c_1=c_2=c_3=0 \quad \text{  and } \quad b=1.
$$ 
We obtain \eqref{eqn:6a13pl1}. 

If only one plane is defined over $\bR$, the three planes on $X_{\bC}$ can be given by 
$$
\Pi_1=\{x_1=x_2=0\},
$$
$$
\quad\Pi_2'=\{x_1=x_3+ix_4=0\},\quad\Pi_3'=\{x_1=x_3-ix_4=0\},
$$
and the singular points of $X_{\bC}$ are
$$
p_1=[0:0:1:i:0],\quad p_3=[0:1:0:0:i],\quad p_5=[0:0:1:-i:0],
$$
$$
p_2=[0:0:1:i:i],\quad p_4=[0:1:0:0:-i],\quad p_6=[0:0:1:-i:-i].
$$
Similarly as above, one can check that $X$ is given by \eqref{eqn:6a13pl2}.
\end{proof}

Projecting from $\Pi_1,$ the projection map 
\begin{align}\label{eqn:proj6a1}
    \pi: X\to \bP^1, \quad (\mathbf x)\mapsto (x_1, x_2),
\end{align}
endows $X$ with the structure of a quadric surface bundle.

\begin{prop}
    Let $X$ be a cubic  given by \eqref{eqn:6a13pl1}. Put 
    $$
    \Delta_1(t_1,t_2,t_3,x):= x^4 + t_1x^3 + (a - 4)x^2 -(4t_1 + t_2t_3)x + (-4a + t_2^2 + t_3^2).
    $$
    Then one of the following holds:
    \begin{enumerate}
        \item When at least one of the following polynomials (in $x_2$)
        $$
        \Delta_1(a_1,a_2,a_3,x_2),\quad \Delta_1(a_2,a_1,a_3,x_2), \quad \Delta_1(a_3,a_2,a_1,x_2)
        $$ 
        has real roots, then $X$ is rational.
        \item When none of the polynomials (in $x_2$)
        $$
        \Delta_1(a_1,a_2,a_3,x_2),\quad \Delta_1(a_2,a_1,a_3,x_2), \quad \Delta_1(a_3,a_2,a_1,x_2)
        $$ 
        has a real root, then $X(\bR)$ is connected and contains no line disjoint from any of the planes $\Pi_i, i=1,2,3$ in $X.$
    \end{enumerate}
\end{prop}
\begin{proof}
    In the affine chart $\{x_1=1\},$ the general fiber $\pi^{-1}(x_2)$ is isomorphic to the quadric surface 
    \begin{align}\label{eqn:6a1quaform}
          \{u_1^2+(1-\frac{x_2^2}{4})u_2^2+u_3^2+ \frac{\Delta_1(a_1,a_2,a_3,x_2)}{x_2^2-4} u_4^2=0\}\subset\bP^3_{u_1,\ldots,u_4}.
    \end{align}
 When $\Delta_1(a_1,a_2,a_3,x_2)$, as a polynomial in $x_2$, has a real root $t$, the fiber~\eqref{eqn:6a1quaform} above $x_2=t$ is a quadric cone (including the cases $t=\pm2$), which contains a line intersecting $\Pi_1$ at one point.  Thus, the line is disjoint from $\Pi_2$. It follows that $X$ is rational. By symmetry, when $\Delta_1(a_2,a_1,a_3,x_2)$ (respectively,  $\Delta_1(a_3,a_2,a_1,x_2)$) has a real root, there exists a line in $X$ disjoint from $\Pi_2$ (respectively, $\Pi_3$), and $X$ is rational.  
    
    When none of these three polynomials has a real root, the image of $\pi$ is connected but none of the fibers of $\pi$ is a hyperboloid. Therefore, $X(\bR)$ is connected but contains no line disjoint from $\Pi_1$, $\Pi_2$ or $\Pi_3$.  
\end{proof}

\begin{prop}
    Let $X$ be a cubic given by \eqref{eqn:6a13pl2}, and
    $$
    \Delta_2(x):=\frac{a_2^2}{4}x^4 -(a+a_2^2+a_3^2)x^3 + (4a - a_1)x^2 + (4a_1 - 1)x + 4\in \bR[x].
    $$
Then
    $X(\bR)$ is disconnected if and only if one of the following holds:
    \begin{enumerate}
        \item $\Delta_2$ has 2 distinct real roots $\alpha_1$ and $\alpha_2$, and $0<\alpha_1<\alpha_2<4$,
        \item $\Delta_2$ has 4 distinct real roots $\alpha_1<\alpha_2<\alpha_3<\alpha_4$ and $0<\alpha_1<\alpha_2<4$,
        \item $\Delta_2$ has 4 distinct real roots $\alpha_1<\alpha_2<\alpha_3<\alpha_4$ and $\alpha_1<0<\alpha_3<\alpha_4<4$.
    \end{enumerate}
\end{prop}
\begin{proof}
  In the affine chart $\{x_2=1\}$, when $x_1\ne 0,4$, the fiber $\pi^{-1}(x_1)$ is isomorphic to the quadric surface given by 
    $$
    S=\{u_1^2+u_2^2+ru_3^2+\frac{x_1^2}{4r}\cdot \Delta_2(x_1)u_4^2=0\}\subset\bP^3_{u_1,\ldots,u_4}, 
    $$
    where
    $$
    r=-\frac{x_1(x_1-4)}{4}.
    $$
One can check that the fiber $\pi^{-1}(x_1)$ is empty if and only if $r>0$ and $\Delta_2(x_1)>0$. Since $X$ has no real singular points and $X_\bC$ only contains three planes, $\Delta_2(x_1)$ can only have simple roots. It follows that $X(\bR)$ is disconnected in the cases described in the statement.

\end{proof}

\subsection*{$2\sA_2+4\sA_1$ and $2\sA_3+4\sA_1$-singularities
}  

\begin{prop}
\label{sect:6rat}
    Let $X\subset \bP^4$ be a cubic such that $X_\bC$ has singularities of type
    $$
2\sA_2+4\sA_1, \quad \text{ or} \quad 2\sA_3+4\sA_1.
    $$
    Then $X$ is rational over $\bR$. 
\end{prop}
\begin{proof}
    Geometrically, the four $\sA_1$-points lie on one plane and the line passing through the other two singular points is disjoint from this plane, see \cite[Section 9]{CMTZ}. The plane and the line are defined over $\bR$; thus, $X$ is rational over $\bR$.
\end{proof}

\section{Eight singular points}
\label{sect:eight}

\begin{prop}
\label{prop:8-rat}
    Let $X$ be a cubic without real singular points and such that $X_{\bC}$ has $8\sA_1$-singularities. Then $X$ is (stably) rational over $\bR$ if and only if it satisfies {\bf (H1)} if and only if $X$ contains three planes over $\bR$. 
\end{prop}

\begin{proof}
We recall the geometry of $X$, over $\bC$: it has five planes $\Pi_1,\ldots,\Pi_5$ and 8 singular points $p_1.\ldots,p_5$ with the following inclusion relations:
\begin{align*}
    \Pi_1\supset\{p_1,p_2,p_6,p_8\},\\
    \Pi_2\supset\{p_1,p_2,p_5,p_7\},\\
    \Pi_3\supset\{p_5,p_6,p_7,p_8\},\\
    \Pi_4\supset\{p_3,p_4,p_5,p_6\},\\
    \Pi_5\supset\{p_3,p_4,p_7,p_8\}.
\end{align*}
Enumerating involutions $\iota\in\fS_8$ preserving the configuration of singular points and planes as above, we find that $\Gamma$ acts on the eight singular points via one of the following possibilities for $\iota$:
\begin{enumerate}
    \item $\iota=(1, 2)(3, 4)(5, 6)(7, 8)$: the planes $\Pi_3, \Pi_4, \Pi_5$ are defined over $\bR$ and the line through $p_1$ and $p_2$ is disjoint from $\Pi_4$ (cf. \cite[Section 8]{CTZ}). It follows that $X$ is rational over $\bR$.
     \item $\iota=(1, 2)(3, 4)(5, 7)(6, 8)$: the planes $\Pi_1, \Pi_2, \Pi_3$ are defined over $\bR$ and the line through $p_3$ and $p_4$ is disjoint from $\Pi_1$, thus $X$ is rational over $\bR.$
     \item $\iota= (1, 2)(3, 4)(5, 8)(6, 7)$: only $\Pi_3$ is defined over $\bR.$ One can compute (cf. \cite[Section 8]{CTZ}) that
     $$
\rH^1(\langle\iota\rangle,\Pic(\tilde X_\bC))=\bZ/2.
$$
It follows that $X$ is not stably rational over $\bR.$
\end{enumerate}
\end{proof}

\begin{rema}
All three cases in the proof of Proposition~\ref{prop:8-rat} can be indeed realized, e.g.,
    \begin{enumerate}
    \item $ a_1(x_1^2 + x_3^2)x_4 + x_4x_5^2+a_2(
    x_2^2x_5 - x_3^2x_5 + x_4^2x_5)+a_3
    x_3x_4x_5=0$,
    \item $ a_1(x_1^2 - x_3^2)x_4 + x_4x_5^2+a_2(
    x_2^2x_5 + x_3^2x_5 + x_4^2x_5)+a_3
    x_3x_4x_5=0$,
        \item $ a_1(x_1^2 + x_3^2)x_4 + x_4x_5^2+a_2(
    x_2^2x_5 + x_3^2x_5 + x_4^2x_5)+a_3
    x_3x_4x_5=0$,
    \end{enumerate}
    with general parameters $a_1,a_2,a_3\in \bR.$
\end{rema}

 \section{Non-isolated Singularities}
 
Let $X$ be a real cubic threefold such that $X_\bC$ has non-isolated singularities. As before, we assume that 
$X^{\rm{sing}}(\bR)=\emptyset$ and $X(\bC)$ is not a cone. It easily follows from 
\cite[Proposition 4.2]{yokocub} that $X^{\rm{sing}}$ is either a pointless conic or a pointless form of the rational normal quartic curve. We proceed to analyze these cases. 

 \subsection*{
 $X^{\rm{sing}}$ is a smooth conic}
Over $\bR$, the normal form of such $X$ is 
\begin{align}\label{eqn:singconic}
c+x_1q_1+x_2q_2+x_3q_3+x_4(x_1^2+x_2^2+x_3^2)=0,
\end{align}
where 
\begin{align*}
c&=a_1x_4^3+a_2x_4x_5^2+a_3x_4^2x_5+a_4x_5^3,\\
  q_1&=a_5x_4^2+a_6x_4x_5+a_7x_5^2,\\
  q_2&=a_8x_4^2+a_9x_4x_5+a_{10}x_5^2,\\
  q_3&=a_{11}x_4^2+a_{12}x_4x_5+a_{13}x_5^2,
\end{align*}
 with general parameters $a_1,\ldots,a_{13}\in\bR$. We have 
 $$
 X^{\rm{sing}}=\{x_1^2+x_2^2+x_3^2=x_4=x_5=0\}\subset\bP^4.
 $$
 Note that $X$ contains the plane $\Pi=\{x_4=x_5=0\}$. As before, projection from $\Pi$ 
yields a quadric surface bundle 
 $$
 \pi: X\to\bP^1_{x_4,x_5}.
 $$
\begin{prop}
    Let $X$ be a cubic given by \eqref{eqn:singconic}. Then $X(\bR)$ is disconnected if and only if $\Delta(x)$ has four distinct real roots, where $\Delta(x)$ is the quartic polynomial given by 
    \begin{multline*}
 -((a_7+a_{10})^2 + a_{13}^2)x^4+ (4a_4 - 2(a_6+a_9)(a_7+a_{10})- 2a_{12}a_{13})x^3+\\ + (4a_2 - 2(a_5+a_8)(a_7 +a_{10}) - (a_6+a_9)^2 - 2a_{11}a_{13} - a_{12}^2)x^2 +\\+ (4a_3 - 2(a_5+a_8)(a_6+a_9)- 2a_{11}a_{12})x +\\+ 4a_1 - (a_5+a_8)^2 -
    a_{11}^2
    \end{multline*}
\end{prop}
\begin{proof}
    Set $x_4=1$. The fiber $\pi^{-1}(x_5)$ is isomorphic to the quadric surface given by 
    $$
    \{u_1^2+u_2^2+u_3^2+\Delta(x_5)u_4^2=0\}\subset \bP^3_{u_1,u_2,u_3,u_4}.
    $$
    It is clear that its real locus is empty if and only if $\Delta(x_5)>0$. It follows that $X(\bR)$ is disconnected if and only if $\Delta\in \bR[x]$ as a quartic polynomial has four distinct real roots.
\end{proof}
 \begin{rema}
   When $X(\bR)$ is connected, the rationality of $X$ is open, see \cite{CTA}. We provide examples with $X(\bR)$  
   \begin{itemize}
        \item   disconnected: $
    a_1=a_2=a_3=100, a_4=5, a_5=20, a_6=a_7=a_8=\ldots=a_{13}=1.
    $
    \item  connected: $a_4=-1,a_2=1, a_1=a_3=a_5=a_6=a_7=a_8=\ldots=a_{13}=0.$ 
    \end{itemize}
    The connected example also appeared in \cite[Section 11.3]{CTA}.  
 \end{rema}
 
\subsection*{Chordal cubic}
A cubic $X_{\bC}$ singular along a rational normal quartic curve $C$ is unique: it is the secant variety of $C$, also known as the {\em chordal cubic}, see, e.g., \cite{Allcub} or \cite{yokocub}. Over $\bR$, and 
under the assumption $C(\bR)=\emptyset$,
there is a unique such cubic $X$,  given by 
\begin{align}\label{eqn:chordal}
    \{x_1^2x_2 + x_1x_2^2 + x_2x_3^2 + x_1x_4^2 - 2x_3x_4x_5 - 
    (x_1+x_2)x_5^2=0\}.
\end{align}

\begin{prop}
    The chordal cubic $X$ given by \eqref{eqn:chordal} is rational.
\end{prop}
\begin{proof}
 Let $C=X^{{\rm sing}}$. The linear system $|\cO_X(2)-C|$ has dimension 5 and gives a contraction 
$$
\varphi:X\dashrightarrow S,
$$
where $S=\bP^2\hookrightarrow \bP^5$ is the Veronese surface. Projection from a general plane $\Pi$ in $\bP^4$ gives a rational map
$$
\pi: X\dashrightarrow \bP^1.
$$
The product 
$$
\pi\times\varphi: X\dashrightarrow\bP^1\times\bP^2
$$
is a birational map. It follows that $X$ is rational over $\bR$. 
\end{proof}

\bibliographystyle{alpha}
\bibliography{cuberat}

\end{document}